\documentclass[ 
showkeys,
aps,
12pt,
]{revtex4-2}

\usepackage{verbatim}

\newcommand{\Aut}{\ensuremath{\operatorname{Aut}}}

\newcommand{\id}{\ensuremath{\text{\rm id}}}

\usepackage[dvipsnames]{xcolor}
\usepackage{amsmath, amsthm, amscd, amsfonts, amssymb, graphicx, color, soul, enumerate, xcolor,  mathrsfs, latexsym}
\usepackage[bookmarksnumbered, colorlinks, plainpages,backref]{hyperref}
\hypersetup{colorlinks=true,linkcolor=purple, anchorcolor=green, citecolor=cyan, urlcolor=blue, filecolor=magenta, pdftoolbar=true}

\usepackage[all]{xy}

\usepackage{tikz}
\usetikzlibrary{calc,decorations.pathreplacing,decorations.markings,positioning,shapes}

\definecolor{cupgreen}{rgb}{0,0.498,0.208}
\definecolor{cupblue}{rgb}{0,0,.5}
\definecolor{cupred}{rgb}{1,0.04,0}
\definecolor{cuppink}{rgb}{0.925,0,0.545}
\definecolor{cupmagenta}{rgb}{0.624,0.161,0.424}
\definecolor{cupbrown}{rgb}{0.71,0.212,0.133}

\definecolor{cupgreen}{rgb}{0,0,0}
\definecolor{cupblue}{rgb}{0,0,0}
\definecolor{cupred}{rgb}{0,0,0}
\definecolor{cuppink}{rgb}{0,0,0}
\definecolor{cupmagenta}{rgb}{0,0,0}
\definecolor{cupbrown}{rgb}{0,0,0}

\definecolor{TITLE}{rgb}{0,0,0}
\definecolor{AUTHOR1}{rgb}{0.00,0.59,0.00}
\definecolor{AUTHOR2}{rgb}{0.50,0.00,1.00}
\definecolor{SECTION}{rgb}{0.50,0.00,1.00}
\definecolor{THM}{rgb}{0.8,0,0.1}
\definecolor{SEC}{rgb}{0,0,1}

\normalsize
\newtheorem{theorem}{{\color{THM} Theorem}}[section]

\newtheorem{lemma}{{\color{THM} Lemma}}[section]

\theoremstyle{definition}

\newtheorem{example}[theorem]{{\color{THM}Example}}

\numberwithin{equation}{section}

\DeclareRobustCommand{\stirling}{\genfrac\{\}{0pt}{}}

\begin{document}
	\title{The number of distinguishing colorings  of a Cartesian product graph}
	
	\author{Saeid Alikhani}
	\email{alikhani@yazd.ac.ir}
	\affiliation{Department of Mathematical Sciences, Yazd University, 89195-741, Yazd, Iran.}
	
	\author{Mohammad H. Shekarriz}
	\email{mhshekarriz@yazd.ac.ir}
	\affiliation{Department of Mathematical Sciences, Yazd University, 89195-741, Yazd, Iran.}	
	
	\vspace{2cm} 
	\bigskip
	\begin{abstract}
		A vertex coloring is called distinguishing if the identity is the only automorphism that can preserve it. The distinguishing threshold $\theta(G)$ of a graph $G$ is the minimum number of colors $k$ required that any arbitrary $k$-coloring of $G$ is distinguishing. In this paper, we calculate the distinguishing threshold of a Cartesian product graph. Moreover, we calculate the number of non-equivalent distinguishing colorings of grids.
		\vspace{0.3cm}
		
		\noindent Mathematics Subject Classification (2020):  05C15, 05C76.
	\end{abstract}
	
	\keywords{distinguishing coloring, Cartesian product, distinguishing threshold, holographic coloring}

	\maketitle

	\section{Introduction}
	A (vertex, edge, total, proper, etc.) coloring of a graph $G$ is called distinguishing (or symmetry breaking) if no non-identity automorphism of $G$ preserves it, and the distinguishing number, denoted by $D(G)$, is the smallest number of colors required for vertex coloring to be distinguishing. This terminology has been introduced by Albertson and Collins~\cite{albertson1996symmetry} in 1996 and initiated many results and generalizations. The concept is however at least two decades older, but it was called \emph{asymmetric coloring}, e.g. see \cite{Babai1977} by Babai.
	
	When vertex coloring of a graph $G$ is distinguishing, we say that this coloring \emph{breaks} all the symmetries of $G$. A $k$-\textit{distinguishing coloring} is a coloring that uses exactly $k$ colors. For a positive integer $d$, a graph $G$ is  {\it $d$-distinguishable} if there exists a distinguishing vertex coloring with $d$ colors. The distinguishing number of some classes of graphs are as follows: $D(K_{n})=n$, $D(K_{n,n})=n+1$, $D(P_n) = 2$ for $n\geq 2$, $D(C_3)=D(C_4)=D(C_5)=3$ while $D(C_n)=2$ for $n\geq 6$ \cite{albertson1996symmetry}.
	
	Because many graphs in applications appear to be a product of smaller graphs, one interesting natural question for every graph theoretical index is to consider it for product graphs, especially for Cartesian products. For example, Bogstad and Cowen \cite{Bogstad2004} proved that for $k \geq 4$, every hypercube $Q_k$ of dimension $k$ 
	is $2$-distinguishable. Moreover, Imrich and Klav{\v{z}}ar \cite{Imrich2006CartPower} showed that the distinguishing number of Cartesian powers of a connected graph $G$ is equal to two except for $K_2^2, K_3^2, K_2^3$. Meanwhile, Imrich, Jerebic, and Klav\v{z}ar \cite{Imrich2008CartComp} proved that Cartesian products of relatively prime graphs whose sizes are close to each other can be distinguished with a small number of colors. 
	
	Several indices, relevant to the distinguishing number, were introduced by Ahmadi, Alinaghipour, and Shekarriz~\cite{ahmadi2020number}. Two colorings $c_1$ and $c_2$ of a graph $G$ are \emph{equivalent} if there is an automorphism $\alpha$ of $G$ such that $c_1 (v) = c_2 (\alpha(v) )$ for all $v\in V (G)$. The number of non-equivalent distinguishing colorings of a graph $G$ with $\{1,\ldots,k\}$ (as the set of admissible colors) is denoted by $\Phi_k (G)$, while the number of non-equivalent $k$-distinguishing colorings of a graph $G$ with $\{1,\ldots,k\}$ is denoted by $\varphi_k (G)$. Evidently, when $G$ has no distinguishing colorings with exactly $k$ colors, we have $\varphi_k (G) =0$. For a graph $G$, the \textit{distinguishing threshold}, $\theta(G)$, is the minimum number $t$ such that for any $k\geq t$, any arbitrary $k$-coloring of $G$ is distinguishing \cite{ahmadi2020number}. 
	
	Shekarriz et al.~\cite{Shekarriz2023theta} alternatively defined the distinguishing threshold terms of $\mathrm{Aut}(G)$. For a non-identity automorphism $\alpha$, let $c(\alpha)$ be the number of cycles of $\alpha$ as a permutation and put $c(\id)=0$. Then the distinguishing threshold of a graph $G$ is 
	\begin{equation}\label{max-lem}
		\theta(G) = 1 + \max\left\{c(\alpha) \; :\; \alpha\in\mathrm{Aut}(G)\right\}.
	\end{equation}
	
	To date, the distinguishing threshold has been studied for the Johnson graphs~\cite{Shekarriz2023theta}, the corona product, vertex-sum, rooted product, and lexicographic product~\cite{Shekarriz2023vsum}. It is shown to be useful to calculate the distinguishing number of disconnected graphs and the lexicographic product~\cite{ahmadi2020number,Shekarriz2023vsum}. In this paper, we consider the distinguishing threshold for the Cartesian products.
	
	Here in this paper, we want to consider the following questions. Let $G$ and $H$ be two prime graphs. What can be said about $\theta (G \Box H)$ in terms of $\theta (G)$ and $\theta (H)$? Moreover, what can be said about $\varphi (G \Box H)$ or $\Phi_k (G \Box H)$ in terms of relevant indices of factors? If it is not possible to answer these questions rapidly, can we find an explicit formula for these indices when we limit our attention to paths and cycles?
	
	Some preliminaries of the Cartesian product, distinguishing coloring and the distinguishing threshold, are introduced in Section \ref{preliminaries}. In Section \ref{Lem-Sec}, a generalization of Lemma 4.1 by Gorzkowska and Shekarriz~\cite{ola2018} is presented. The lemma gives a necessary and sufficient condition on the coloring of the Cartesian product of graphs to be distinguishing coloring. Afterwards, we consider $\theta (G\Box H)$, $\Phi_k(G\Box H)$ and $\varphi_k (G\Box H)$ in Sections \ref{main1} and \ref{main2}.
	
	\section{Preliminaries}\label{preliminaries}
	
	\subsection{Distinguishing indices}
	
	The number of non-equivalent distinguishing colorings of a graph $G$ with $\{1, \ldots, k\}$ as the set of admissible colors, $\Phi_k(G)$, can be calculated if we know the number of non-equivalent $i$-distinguishing colorings of $G$, $\varphi_i(G)$, for all $i\in\{1,\ldots,k\}$. The relation is given as follows~\cite{ahmadi2020number}:
	\begin{equation}\label{Phi-phi}
		\Phi_k(G) = \sum_{i=D(G)}^k {k\choose i} \varphi_i(G).
	\end{equation}
	
	Moreover, one can easily verify that for $n,k \geq 1$, we have $\Phi_k(P_n) = \frac{1}{2}(k^n - k^{\lceil \frac{n}{2} \rceil})$ and for $n\geq 2$ and $k \geq n$ we have $\Phi_k(K_n) = {k \choose n}$ \cite{ahmadi2020number}. The following theorem, in which the notation $\stirling{n}{k}$ denotes the Stirling number of the second kind, reveals the importance of the distinguishing threshold.
	
	\begin{theorem}\label{Aut(G)}  \textnormal{\cite{ahmadi2020number}}
		Let $G$ be a graph on $n$ vertices. For any $k\geq \theta(G)$ we have
		\[\varphi_k (G)=\frac{k! \stirling{n}{k}}{|\Aut(G)|}.\]
	\end{theorem}
	
	The distinguishing threshold for some classes of graphs is already calculated. For example, we have $\theta(P_m)=\lceil\frac{m}{2}\rceil+1$ and $\theta(C_n)=\lfloor\frac{n}{2}\rfloor+2$ for $m\geq 2$ and $n\geq 3$~\cite{ahmadi2020number}.
	
	
	

	
	Shekarriz et al. also proved that the only graphs with the distinguishing threshold 2 are $K_2$ and $\overline{K}_2$. Moreover, they studied graphs whose distinguishing thresholds are 3 and calculated the threshold for graphs in the Johnson scheme~\cite{Shekarriz2023theta}. Furthermore, in another paper, Shekarriz et al. calculated the threshold for some product graphs, such as the corona and the lexicographic products~\cite{Shekarriz2023vsum}.

	\subsection{The Cartesian product of graphs}
	
	The {\it Cartesian product} of graphs $G$ and $H$ is a graph, denoted by $G \Box H$, whose vertex set is $V(G) \times V(H)$, and two vertices $(g, h)$, $(g', h')$ are adjacent if either $g = g'$ and $h h' \in E(H)$, or $g g' \in E(G)$ and $h = h'$. Usually $G \Box G$ is shown by $G^2$ and the {\it $k$-th Cartesian power} of $G$ is $G^k = G \Box G^{k-1}$. A graph $G$ is {\it prime} if it cannot be represented as the Cartesian product of two graphs non-isomorphic with $G$. Two graphs $G$ and $H$ are {\it relatively prime} if they do not have a common non-trivial factor \cite{HIK2011}. 
	
	Every graph has a unique prime factorization with respect to the Cartesian product~\cite{HIK2011}. When factorization is known, the automorphism group can also be expressed by the following theorem by Imrich (and independently by Miller).
	
	\begin{theorem} \textnormal{\cite{HIK2011}} \label{autG}
		Suppose $\psi$ is an automorphism of a connected graph $G$ with prime factor decomposition $G = G_1 \Box G_2 \Box \cdots \Box G_k$. Then there is a permutation $\pi$ of the set $\{1, 2, \dots , k\}$ and there are isomorphisms $\psi_i \colon G_{\pi(i)} \mapsto G_i$, $i=1, \dots, k$, such that
		$$\psi(x_1, x_2, \dots, x_k) = (\psi_1 (x_{\pi(1)}), \psi_2 (x_{\pi(2)}), \dots, \psi_r (x_{\pi(k)})).$$
	\end{theorem}
	
	As a result of Theorem \ref{autG}, when $G=G_1 \Box G_2 \Box \cdots \Box G_k$, the product group $\oplus_{i=1}^{k} \mathrm{Aut}(G_i)$ is a subgroup of $\mathrm{Aut}(G)$. Here, we denote this subgroup by $\mathrm{Aut}^{F}(G)$.
	
	For each factor $G_i$ let the vertex set be $V(G_i) = \{x_{i1}, x_{i2}, \dots, x_{in_i}\}$, where $n_i=\vert G_i\vert$. Then every vertex of the Cartesian product is of the form $(x_{1j_1},x_{2j_2}, \dots , x_{kj_k})$, where $x_{ij_i} \in V(G_i)$. Two vertices of the Cartesian product form an edge 
	\begin{equation}
		(x_{1j_1}, x_{2j_2}, \dots , x_{kj_k})(x_{1l_1}, x_{2l_2}, \dots , x_{kl_k})
	\end{equation}
	if there exists exactly one index $i=1,\ldots,k$ such that $x_{ij_i} x_{il_i}$ is an edge of the factor graph $G_i$ and $x_{tj_t}=x_{tl_t}$ for all indices $t\neq i$. Given a vertex $v = (v_1 , v_2 , \ldots , v_k )$ of the product $G = G_1 \Box G_2 \Box \cdots \Box G_k$, the \emph{$G_i$-layer through $v$} is the induced subgraph
	\begin{equation}\label{layer_through_v}
		G_{i}^{v}=G \left[ \{ x \in V (G) \ \vert \  p_{j}(x) = v_j \text{ for }j\neq i\}\right],
	\end{equation}
	where $p_j$ is the projection mapping to the $j^\text{th}$-factor of $G$~\cite{HIK2011}.
	
	By \emph{$i^{\text{th}}$-quotient subgraph of $G$} we mean the graph
	\begin{equation}\label{quotient}
		Q_{i}=G \diagup G_{i} \simeq G_1 \Box \cdots \Box G_{i-1} \Box G_{i+1} \Box \cdots \Box G_k.
	\end{equation}
	It is also evident that $G \simeq G_i \Box Q_i$~\cite{HIK2011}.

	\section{The holographic coloring}\label{Lem-Sec}
	
	The material mentioned in this section is mostly a generalization of~\cite[Lemma 4.1]{ola2018}. Here, we need it two-sided, so that it can be used to calculate the number of distinguishing colorings of an arbitrary Cartesian product graph. To do so, we have altered some notions therein as follows. 
	
	Suppose that $G=G_1 \Box G_2 \Box \cdots \Box G_k$ for some $k\geq 2$ is a connected graph decomposed into a prime factorization, and $f$ is a total coloring. Note that a vertex or an edge coloring can be easily transformed into a total coloring (by giving all edges or vertices a fixed color). Therefore, $f$ can also be a vertex or an edge coloring.  
	
	For $i=1,\ldots,k$, let $V(G_{i})=\{ 1_{i},\ldots, m_i \}$ and for each $j=1,...,m$, consider 
	\begin{equation}
		u_j = ( 1_{1},1_{2}, \ldots, 1_{i-1}, j_{i}, 1_{i+1},\ldots, 1_{k} ),
	\end{equation}
	where $1_r$ is the first vertex of $G_r$ in our fixed ordering. It is a vertex of $G$ and we can speak of $G_i^{u_j}$ which is defined in Equation~\ref{layer_through_v}.
	
	Meanwhile, for each $i$, the set of colors that we use are equivalent classes of colored quotients layers $Q_i^v$ under some equivalent relations we introduce here. Suppose that $\alpha\in \mathrm{Aut}(Q_{i})$. We define a map, namely $\varphi_{\alpha}$, from $Q_i^{u_j}$ onto $Q_i^{u_t}$, using $\alpha$ and our fixed ordering of $Q_i$, so that $\varphi_{\alpha}$ maps the vertex of $Q_i^{u_j}$ with the same ordering as $x\in Q_i$ onto the vertex of $Q_i^{u_t}$ with the same ordering as $\alpha(x)$. In this case, it is evident that $\varphi_\alpha$ is an isomorphism and we say it is a \emph{lifting of $\alpha$}. Roughly speaking, the lifting of an automorphism produces an isomorphism from one copy of $Q_i$ onto another.
	
	The coloring $f$ induces a (total) coloring on $Q_i^{u_j}$. This colored graph is denoted here by $\check{Q}_{i}^{u_j}= (Q_{i}^{u_j},f)$. For an automorphism $\alpha\in\Aut(Q_{i})$, we say that the color $\check{Q}_{i}^{u_j}$ is \emph{$\alpha$-equivalent} to $\check{Q}_{i}^{u_t}$ if there is a (total) color-preserving isomorphism $\varphi: Q_{i}^{u_j}\longrightarrow Q_{i}^{u_t}$ which is a lifting of $\alpha$ or $\alpha^{-1}$.
	
	
	Let $e=v_{i}w_i\in E(G_i)$ and $\overline{Q_i^e}$ be the graph isomorphic to $Q_i$ constructed as follows: its vertex set consists of edges of $G$ of the form $$(u_i, x)(v_i, x)=(x_{1},\ldots,x_{i-1}, v_{i}, x_{i+1},\dots , x_{k})(x_{1},\ldots,x_{i-1}, w_{i}, x_{i+1}, \dots , x_{k})$$ for $x\in V(Q_i)$ (i.e. $x_j\in V(G_j)$, for $j\neq i$). If $x$ and $y$ are adjacent in $Q_i$, then vertices $(u_i, x)(v_i, x)$ and $(u_i, y)(v_i, y)$ are adjacent in $\overline{Q_i^e}$. Each vertex $(u_i, x)(v_i, x)$ of $\overline{Q_i^e}$ is an edge of $G$, so it is colored by the total coloring $f$. Therefore, $f$ induces a vertex coloring on $\overline{Q_i^e}$, and this colored graph is denoted here by $\hat{\overline{Q_i^e}}=(\overline{Q_i^e},f)$. Similarly, the colored graph $\hat{\overline{Q_i^e}}$ is \emph{$\alpha$-equivalent} to $\hat{\overline{Q_i^{e'}}}$ if there is a vertex-color-preserving isomorphism $\vartheta:\overline{Q_i^e} \longrightarrow \overline{Q_i^{e'}}$ which is a lifting of $\alpha$ or $\alpha^{-1}$.
	
	For each $i\in\{1,\ldots,k\}$, we can color vertices and edges of $G_i$ by colored graphs $Q_{i}$s as follows; let $G_i^f$ be the total coloring of $G_i$ in which each vertex $j\in V(G_i)$ is colored by $\check{Q}_{i}^{u_j}$ and each edge $e=j {\ell}\in E(G_i)$ is colored by $\hat{\overline{Q_i^{e}}}$. To distinguish this coloring from the similar coloring of~\cite{ola2018}, this total coloring of $G_i$ is called the \emph{holographic coloring of $G_i$ induced by $f$}. Finally, two coloring $G_i^f$ and $G_j^f$ are equivalent if there is a total-color-preserving isomorphism from one to another. In this case we write $G_i^f \simeq G_j^f$
	
	Now, we can articulate the desired statement. 
	
	\begin{lemma}\label{AUL2}
		Let $k\geq 2$ and $G=G_1 \Box G_2 \Box \cdots \Box G_k$ be a connected graph decomposed into Cartesian prime factors. A (total) coloring $f$ is a distinguishing coloring for $G$ if and only if for each $i=1,\ldots, k$  we have
		
		\begin{itemize}
			\item[i.] $G_i^f \not\simeq G_j^f$ for all $j=1,\ldots,k$ such that $j\neq i$, and
			\item[ii.] for each $\alpha\in \mathrm{Aut}(Q_i)$ and for each non-identity $\beta\in\Aut(G_i)$, there is a vertex $v\in G_i$ or an edge $e\in E(G_i)$ such that either $\check{Q}_i^v$ and $\check{Q}_i^{\beta(v)}$ or $\overline{Q_i^e}$ and $\overline{Q_i^{\beta(e)}}$ are not $\alpha$-equivalent.
		\end{itemize}
	\end{lemma}
	\begin{proof}
		First, suppose that for an $i=1,\ldots,k$ and an automorphism $\alpha\in \mathrm{Aut}(Q_i)$ there is a $\beta\in\Aut(G_i)$ such that $\check{Q}_i^v$ and $\check{Q}_i^{\beta(v)}$ are $\alpha$-equivalent and $\overline{Q_i^e}$ and $\overline{Q_i^{\beta(e)}}$ are also $\alpha$-equivalent. Then there is an automorphism of $G$, of the form $\varphi=(\alpha,\beta)\in\mathrm{Aut}(Q_i)\oplus\mathrm{Aut}(G_i)$, such that it preserves $f$. It is also evident that if we have $G_i^f \simeq G_j^f$ for some $j\neq i$, then the transposition of factors $i$ and $j$ is a non-identity automorphism of $G$ which also preserves $f$.
		
		Conversely, suppose that both items $i$ and $ii$ above are true. If $\varphi:G\longrightarrow G$ is a color-preserving automorphism, then by Theorem \ref{autG}, there is a permutation $\pi$ of the set $\{1, 2, \dots , k\}$ and there are isomorphisms $\psi_i \colon G_{\pi(i)} \mapsto G_i$, $i=1, \dots, k$, such that
		$$\varphi(x_1, x_2, \dots, x_k) = (\psi_1 (x_{\pi(1)}), \psi_2 (x_{\pi(2)}), \dots, \psi_k (x_{\pi(k)})).$$ Since we have $G_i^f \not\simeq G_j^f$ for all $j\neq i$, it can be deduced that $\pi$ is the identity permutation. Thus, $\varphi=(\psi_1,\ldots,\psi_k)\in\oplus_{i=1}^{k} \mathrm{Aut}(G_i)$, or equivalently, it is of the form $\varphi=(\alpha,\beta)\in\mathrm{Aut}^{F}(Q_i)\oplus\mathrm{Aut}(G_i)$. Now, we must have $\varphi=(\alpha,\beta)=(\id_{Q_i},\id_{G_i})=\id_G$ because else the condition of item $ii$ above says that $\varphi$ cannot preserve $f$. Consequently, the only automorphism of $G$ that can preserve $f$ is the identity.
	\end{proof}
	
	As noted at the beginning of this section, Lemma \ref{AUL2} must also be considered true whenever $f$ is a vertex or edge coloring. However, when $f$ is a vertex coloring of $G$, it is redundant to color each edge $e$ of $G_i$ with $\hat{\overline{Q_i^{e}}}$, because all edges of $G_i$ will receive the same color.

	It should also be noted that Lemma \ref{AUL2} remains true if the condition {\it ii} is true for each $\alpha\in \mathrm{Aut}^{F}(Q_i)$ instead of for each $\alpha\in \mathrm{Aut}(Q_i)$. As noted, $\mathrm{Aut}^{F}(Q_i)$ is a subgroup of $\mathrm{Aut}(Q_i)$ and the condition $G_i^f \not\simeq G_j^f$ for all $j\neq i$ makes it redundant to check the lemma for $\alpha\in \mathrm{Aut}(Q)\setminus \mathrm{Aut}^{F}(Q)$.


	The following examples illustrate the holographic coloring and Lemma \ref{AUL2} interpretations.
	
	\begin{example}\label{IdNotEnough}
		Let $G_1 = P_4 \Box P_5$ have a vertex coloring $f$ that makes $(2, 2)$ and $(3,4)$ red while all other vertices of $G_1$ are black, see Figure \ref{id}. This coloring is not a distinguishing one, however, we have $P_4^f \not\simeq P_5^f$ and for each pair of vertices $u,v\in P_4$, $P_4^u$ and $P_4^v$ are not $\id$-equivalent. Similarly, for each pair of vertices $u',v'\in P_5$, $P_5^{u'}$ and $P_5^{v'}$ are not $\id$-equivalent. Thus, it is not adequate to only consider one automorphism of $Q_i$ in item $ii$ of Lemma \ref{AUL2}.
	\end{example}
	
	\begin{figure}
		\[ \xygraph{
			!{<0cm,0cm>; <1cm,0cm>:<0cm,1cm>::}
			!{(1,1)}*{\bullet}="v1" !{(2,1)}*{\bullet}="v2" !{(3,1)}*{\bullet}="v3" !{(4,1)}*{\bullet}="v4" !{(5,1)}*{\bullet}="v5"
			!{(1,2)}*{\bullet}="v6" !{(2,2)}*{\bullet}="v7" !{(3,2)}*{\bullet}="v8" !{(4,2)}*{\color[rgb]{1,0,0}\bullet}="v9" !{(5,2)}*{\bullet}="v10"
			!{(1,3)}*{\bullet}="v11" !{(2,3)}*{\color[rgb]{1,0,0}\bullet}="v12" !{(3,3)}*{\bullet}="v13" !{(4,3)}*{\bullet}="v14" !{(5,3)}*{\bullet}="v15"
			!{(1,4)}*{\bullet}="v16" !{(2,4)}*{\bullet}="v17" !{(3,4)}*{\bullet}="v18" !{(4,4)}*{\bullet}="v19" !{(5,4)}*{\bullet}="v20"
			"v1"-@[black] "v2" "v2"-@[black] "v3" "v3"-@[black] "v4" "v4"-@[black] "v5"
			"v6"-@[black] "v7" "v7"-@[black] "v8" "v8"-@[black] "v9" "v9"-@[black] "v10"
			"v11"-@[black] "v12" "v12"-@[black] "v13" "v13"-@[black] "v14" "v14"-@[black] "v15"
			"v16"-@[black] "v17" "v17"-@[black] "v18" "v18"-@[black] "v19" "v19"-@[black] "v20"
			"v1"-@[black] "v6" "v6"-@[black] "v11" "v11"-@[black] "v16"
			"v2"-@[black] "v7" "v7"-@[black] "v12" "v12"-@[black] "v17"
			"v3"-@[black] "v8" "v8"-@[black] "v13" "v13"-@[black] "v18"
			"v4"-@[black] "v9" "v9"-@[black] "v14" "v14"-@[black] "v19"
			"v5"-@[black] "v10" "v10"-@[black] "v15" "v15"-@[black] "v20"
		} \]
		\caption{The graph $G_1= P_4 \Box P_5$ of Example \ref{IdNotEnough}.}\label{id}
	\end{figure}

	\begin{example}\label{AllRedundant}
		Suppose that $G_2=P_5 \Box P_6$. Let $f$ be a vertex coloring of $G_2$ such that $(2,2)$, $(2,3)$, $(2,4)$ and $(4,5)$ are red vertices while other vertices of $G_2$ are black, see Figure \ref{all}. This coloring is a distinguishing coloring. However, if we have defined equivalence of colors so that two holographic colors $\check{Q}_{i}^{u_j}$ and $\check{Q}_{i}^{u_t}$ are \emph{equivalent} if there is a vertex-color-preserving isomorphism $\varphi: Q_{i}^{u_j}\longrightarrow Q_{i}^{u_t}$, then Lemma \ref{AUL2} would have implied that $f$ is not distinguishing. This shows that we have to consider each $\alpha\in \mathrm{Aut}(Q)$ separately.
		
	\end{example}
	
	\begin{figure}
		\[ \xygraph{
			!{<0cm,0cm>; <1cm,0cm>:<0cm,1cm>::}
			!{(1,1)}*{\bullet}="v1" !{(2,1)}*{\bullet}="v2" !{(3,1)}*{\bullet}="v3" !{(4,1)}*{\bullet}="v4" !{(5,1)}*{\bullet}="v5"	!{(6,1)}*{\bullet}="v6" 
			!{(1,2)}*{\bullet}="v7" !{(2,2)}*{\bullet}="v8" !{(3,2)}*{\bullet}="v9" !{(4,2)}*{\bullet}="v10"
			!{(5,2)}*{\color[rgb]{1,0,0}\bullet}="v11" !{(6,2)}*{\bullet}="v12" 
			!{(1,3)}*{\bullet}="v13" !{(2,3)}*{\bullet}="v14" !{(3,3)}*{\bullet}="v15" !{(4,3)}*{\bullet}="v16" !{(5,3)}*{\bullet}="v17" !{(6,3)}*{\bullet}="v18" 
			!{(1,4)}*{\bullet}="v19" !{(2,4)}*{\color[rgb]{1,0,0}\bullet}="v20" !{(3,4)}*{\color[rgb]{1,0,0}\bullet}="v21" !{(4,4)}*{\color[rgb]{1,0,0}\bullet}="v22" !{(5,4)}*{\bullet}="v23" !{(6,4)}*{\bullet}="v24" 
			!{(1,5)}*{\bullet}="v25" !{(2,5)}*{\bullet}="v26" !{(3,5)}*{\bullet}="v27" !{(4,5)}*{\bullet}="v28" !{(5,5)}*{\bullet}="v29" !{(6,5)}*{\bullet}="v30"
			"v1"-@[black] "v2" "v2"-@[black] "v3" "v3"-@[black] "v4" "v4"-@[black] "v5" "v5"-@[black] "v6"
			"v7"-@[black] "v8" "v8"-@[black] "v9" "v9"-@[black] "v10" "v10"-@[black] "v11" "v11"-@[black] "v12" 
			"v13"-@[black] "v14" "v14"-@[black] "v15" "v15"-@[black] "v16" "v16"-@[black] "v17" "v17"-@[black] "v18"
			"v19"-@[black] "v20" "v20"-@[black] "v21" "v21"-@[black] "v22" "v22"-@[black] "v23" "v23"-@[black] "v24"
			"v25"-@[black] "v26" "v26"-@[black] "v27" "v27"-@[black] "v28" "v28"-@[black] "v29" "v29"-@[black] "v30"
			"v1"-@[black] "v7" "v7"-@[black] "v13" "v13"-@[black] "v19" "v19"-@[black] "v25"
			"v2"-@[black] "v8" "v8"-@[black] "v14" "v14"-@[black] "v20" "v20"-@[black] "v26"
			"v3"-@[black] "v9" "v9"-@[black] "v15" "v15"-@[black] "v21" "v21"-@[black] "v27"
			"v4"-@[black] "v10" "v10"-@[black] "v16" "v16"-@[black] "v22" "v22"-@[black] "v28"
			"v5"-@[black] "v11" "v11"-@[black] "v17" "v17"-@[black] "v23" "v23"-@[black] "v29"
			"v6"-@[black] "v12" "v12"-@[black] "v18" "v18"-@[black] "v24" "v24"-@[black] "v30"
		} \]
		\caption{The graph $G_2 = P_5 \Box P_6$ of Example \ref{AllRedundant}.}\label{all}
	\end{figure}
	
	\begin{example}\label{dist-ex}
		Let $G_3 =P_4 \Box P_5$, whose 2-coloring $f$ is shown in Figure \ref{dist-ex-fig}. We prove that this coloring is distinguishing. Obviously, $P_4 \not\simeq P_5$ implies that $P_4^f \not\simeq P_5^f$ and hence item $i$ of Lemma \ref{AUL2} is satisfied automatically for every coloring of $G_3$. On the other hand, we know that the automorphism groups of $P_4$ and $P_5$ have only one non-identity element. Suppose that $\Aut(P_4)=\{\id_{P_4}, \gamma\}$ and $\Aut(P_5)=\{\id_{P_5},\sigma\}$. It is straightforward to check that item $ii$ of Lemma \ref{AUL2} holds when $(\alpha,\beta)$ is $(\id_{P_4},\sigma)$ or $(\id_{P_5},\gamma)$. For $(\alpha,\beta)=(\gamma,\sigma)$ or $(\alpha,\beta)=(\sigma,\gamma)$, we just need to note that $\sigma$ has a fixed vertex, say $v\in P_5$ for which $\check{Q}_1^{v}$ is not $\gamma$-equivalent to itself. Consequently, item $ii$ of Lemma \ref{AUL2} is also met and therefore $f$ is a distinguishing coloring.
	\end{example}

	\begin{figure}
		\[ \xygraph{
			!{<0cm,0cm>; <1cm,0cm>:<0cm,1cm>::}
			!{(1,1)}*{\bullet}="v1" !{(2,1)}*{\bullet}="v2" !{(3,1)}*{\bullet}="v3" !{(4,1)}*{\bullet}="v4" !{(5,1)}*{\bullet}="v5"
			!{(1,2)}*{\bullet}="v6" !{(2,2)}*{\bullet}="v7" !{(3,2)}*{\bullet}="v8" !{(4,2)}*{\color[rgb]{1,0,0}\bullet}="v9" !{(5,2)}*{\bullet}="v10"
			!{(1,3)}*{\bullet}="v11" !{(2,3)}*{\color[rgb]{1,0,0}\bullet}="v12" !{(3,3)}*{\color[rgb]{1,0,0}\bullet}="v13" !{(4,3)}*{\bullet}="v14" !{(5,3)}*{\bullet}="v15"
			!{(1,4)}*{\bullet}="v16" !{(2,4)}*{\bullet}="v17" !{(3,4)}*{\bullet}="v18" !{(4,4)}*{\bullet}="v19" !{(5,4)}*{\bullet}="v20"
			"v1"-@[black] "v2" "v2"-@[black] "v3" "v3"-@[black] "v4" "v4"-@[black] "v5"
			"v6"-@[black] "v7" "v7"-@[black] "v8" "v8"-@[black] "v9" "v9"-@[black] "v10"
			"v11"-@[black] "v12" "v12"-@[black] "v13" "v13"-@[black] "v14" "v14"-@[black] "v15"
			"v16"-@[black] "v17" "v17"-@[black] "v18" "v18"-@[black] "v19" "v19"-@[black] "v20"
			"v1"-@[black] "v6" "v6"-@[black] "v11" "v11"-@[black] "v16"
			"v2"-@[black] "v7" "v7"-@[black] "v12" "v12"-@[black] "v17"
			"v3"-@[black] "v8" "v8"-@[black] "v13" "v13"-@[black] "v18"
			"v4"-@[black] "v9" "v9"-@[black] "v14" "v14"-@[black] "v19"
			"v5"-@[black] "v10" "v10"-@[black] "v15" "v15"-@[black] "v20"
		} \]
		\caption{The graph $G_3 = P_4 \Box P_5$ of Example \ref{dist-ex}.}\label{dist-ex-fig}
	\end{figure}

	\section{The distinguishing threshold of the Cartesian product of prime graphs}\label{main1}
	
	Using Lemma \ref{AUL2} we are able to calculate the distinguishing threshold of the Cartesian product of prime graphs. 
	
	\begin{theorem}\label{theta-cart-1}
		Let $k\geq 2$ and $G=G_1\Box G_2 \Box \cdots \Box G_k$ be a prime factorization to mutually non-isomorphic connected graphs. Then $$\theta (G)=\max \left\{ \left(\theta(G_i)-1\right)\cdot \vert Q_i\vert\; : \; i=1,\dots,k \right\}+1.$$
	\end{theorem}
	
	\begin{proof}
		Choose a non-distinguishing coloring $g$ for $G_i$ with $\theta (G_i)-1$ colors. Then, there is a non-identity $\sigma\in\Aut(G_i)$ such that it preserves $g$, i.e., there are distinct vertices $u_1, u_2 \in G_i$ such that $u_2 = \sigma(u_1)$ and $g(u_1)=g(u_2)$. For a vertex $v\in G_i$ and a fixed ordering of vertices of $Q_i$, if $g(v)=t$, $t=1,\ldots, \theta (G_i)-1$, then color vertices of $Q_i^v$ by colors $t\cdot1,\ldots, t\cdot\vert Q_i \vert $ keeping the same ordering of vertices in their colors. Call the resulting coloring $f$. Then, $\check{Q}_i^{u_1}$ and $\check{Q}_i^{u_2}$ are $\id_{Q_i}$-equivalent, which means that item $ii$ of Lemma \ref{AUL2} is not met by $f$. Thus, $f$ is not a distinguishing coloring. Consequently, for each $i=1,\ldots, k$ we have $\theta (G)\geq  \left(\theta(G_i)-1\right)\cdot \vert Q_i\vert+1$.
		
		Since for $i\neq j$ we know $G_i$ and $G_j$ are non-isomorphic, for any coloring $f$ of $G$ that we want to know if it is distinguishing, it is obvious that $G_i^f \not\simeq G_j^f$. 
		
		Suppose that $f$ is a coloring of $G$ with $\max \left\{ \left(\theta(G_i)-1\right)\cdot \vert Q_i\vert\; : \; i=1,\dots,k \right\}+1$ colors. For each $\alpha\in \Aut(Q_i)$, and for every $\beta\in \Aut(G_i)$ such that $(\alpha,\beta)\neq (\id_{Q_i},\id_{G_i})$, there are at least $\theta(G_i)$ vertices $u_1,\ldots ,u_{\theta(G_i)}$ such that $\check{Q}_i^{u_j}$ and $\check{Q}_i^{u_k}$, $k\neq j$ are not mutually $\alpha$-equivalent. Therefore, there is at least one $u_j$ among those vertices that $\check{Q}_i^{u_j}$ and $\check{Q}_i^{\beta (u_j)}$ are not $\alpha$-equivalent. Therefore, items $i$ and $ii$ of Lemma \ref{AUL2} are both met and consequently, $f$ is a distinguishing coloring.
	\end{proof}
	
	Theorem \ref{autG} implies that $\Aut(G^k)\cong \mathrm{Sym}(k)\oplus \Aut(G)^k$. We use this fact in the proof of the following theorem.
	
	\begin{theorem}\label{theta-cart-2}
		Let $G$ be a connected prime graph and $k\geq 2$ be a positive integer. Then $$\theta (G^k)=\vert G\vert^{k-1}\cdot\max\left\{ \frac{\vert G\vert +1}{2}, \theta(G)-1\right\}+1.$$
	\end{theorem}
	
	\begin{proof}
		Let $t=\vert G\vert$. First, note that Equation \ref{max-lem} implies that $\theta (G^k) \geq \frac{t^k +t^{k-1}}{2}+1$. This is because there is an automorphism $\beta \in \Aut (G^k)$ which transposes two factors $G_1$ and $G_2$, and its number of cycles is $c(\beta) =\frac{t^k +t^{k-1}}{2}=t^{k-1}\cdot \frac{t +1}{2}$. Similarly, we have $\theta (G^k) \geq t^{k-1}\cdot (\theta(G)-1)+1$, because there is an $\tilde{\alpha} \in \Aut (G^k )$ such that it is a lift from $\alpha \in \Aut (G)$ with $c(\alpha) =\theta(G)-1$. Then, we have $c(\tilde{\alpha}) =t^{k-1}\cdot (\theta (G)-1)$ and consequently, $$\theta (G^k)\geq t^{k-1}\cdot\max\left\{ \frac{t +1}{2}, \theta(G)-1\right\}+1.$$
		
		Now, let $f$ be an arbitrary vertex coloring for $G^k$ with $t^{k-1}\cdot\max\left\{ \frac{t +1}{2}, \theta(G)-1\right\}+1$ colors. From the proof of Theorem \ref{theta-cart-1} we know that for each $i=1,\ldots,k$ the condition $ii$ holds because the number of colors are greater than $$\vert G\vert^{k-1}\cdot \left(\theta(G)-1\right)=  \left(\theta(G_i)-1\right)\cdot\vert Q_i \vert.$$ If for some $i,j\in \{1,\ldots,k\}$, $i\neq j$ we have $G_i^f\simeq G_j^f$, then we must have an isomorphism that maps $G_i^f$ onto $G_j^f$. Whenever this isomorphism maps a vertex $v\in G_i$ onto $u\in G_j$, it maps $\check{Q}_i^v$ onto $\check{Q}_j^u$, which in turn means that the coloring of $Q_i^v$ onto $Q_j^u$ must be equivalent. Therefore, $G_i^f\simeq G_j^f$ implies that there must be a color-preserving automorphism $\gamma$ of $G^k$ such that $c(\gamma) = t^{k-1}\cdot\frac{t +1}{2}$. This is not possible because the number of colors is at least one more than this. Consequently, item $i$ of Lemma \ref{AUL2} is also met and $f$ is a distinguishing coloring.
	\end{proof}
	
	Using the Cartesian prime factorization of a graph and Theorems \ref{theta-cart-1} and \ref{theta-cart-2}, one can easily prove the next theorem, which concludes this section.
	
	\begin{theorem}
		Let $G=G_1^{t_1}\Box G_2^{t_2}\Box \cdots \Box G_k^{t_k}$ be a graph decomposed to its prime factors. Then
		$$\theta(G)=\max\left\{(\theta(G_i^{t_i})-1)\cdot  \frac{\vert G\vert}{\vert G_i^{t_i} \vert} \; : \; i=1,\ldots,k\right\}.$$
	\end{theorem}
	
	\section{Number of non-equivalent distinguishing colorings of grids}\label{main2}
	
	In this section, we calculate $\Phi_k (G\Box H)$ and $\varphi_k (G\Box H)$ when $G$ and $H$ are paths. 
	
	\begin{theorem}
		Let $m,n\geq 2$ be two distinct integers and $k\geq 2$. Then
		$$\Phi_k (P_m \Box P_n)=\frac{1}{4}\left( k^{mn}-k^{m\lceil \frac{n}{2}\rceil}-k^{n\lceil \frac{m}{2}\rceil} -k^{\lceil \frac{mn}{2}\rceil} +2 k^{\lceil \frac{m}{2}\rceil\lceil \frac{n}{2}\rceil}\right).$$
	\end{theorem}
	\begin{proof}
		Since $P_m \not\simeq P_n$, Theorem \ref{autG} implies that $\Aut(P_m\Box P_n)\cong \mathbb{Z}_2 \oplus \mathbb{Z}_2$ has four elements: the identity, a reflection over $P_m$, a reflection over $P_n$ and a rotation of $180^\circ$. Then, for each distinguishing coloring, there are 3 other distinguishing colorings that are equivalent to it. Therefore, the number of non-equivalent distinguishing colorings of $P_m\Box P_n$ is one-fourth of the total number of distinguishing colorings of $P_m\Box P_n$.
		
		The total number of vertex coloring of $P_m\Box P_n$ with at most $k$ colors is $k^{mn}$. The number of colorings that are preserved by the reflection over $P_m$, the reflection over $P_n$, and the rotation of $180^\circ$ are $k^{n\lceil \frac{m}{2}\rceil}$, $k^{m\lceil \frac{n}{2}\rceil}$ and $k^{\lceil \frac{mn}{2}\rceil}$ respectively. These numbers overlap by some colorings, those that are preserved by two out of three non-identity automorphisms, which are $3k^{\lceil \frac{m}{2}\rceil\lceil \frac{n}{2}\rceil}$ colorings. And, there are colorings that are preserved by all three non-identity automorphisms, which are $k^{\lceil \frac{m}{2}\rceil\lceil \frac{n}{2}\rceil}$ ones. Now, the result follows inclusion-exclusion.
	\end{proof}
	
	When the two factors are the same, calculations via inclusion-exclusion are a bit harder as there are automorphisms that exchange the factors. When the size of the automorphism group increases, it becomes hard to apply and validate an inclusion-exclusion argument. However, if we can classify similar states, we can easily use the argument. We do this in the proof of the following theorem.
	
	\begin{theorem}
		Let $n\geq 3$ be an integer and $k\geq 2$. Then $$\Phi_k (P_n^2)=\frac{1}{8}\left( k^{n^2} -k^{\lceil \frac{n^2}{2}\rceil} - 2k^{n\lceil\frac{n}{2}\rceil}-2k^{n\left(\frac{n+1}{2}\right) } +2k^{\lceil \frac{n}{2}\rceil^2}+2k^{\lceil \frac{n}{2}\rceil \lceil \frac{n+1}{2}\rceil}\right).$$
	\end{theorem}
	
	\begin{proof}
		The automorphism group of $G=P_n^2 = P_n\Box P_n$ is isomorphic to the dihedral group $D_8$. It has 8 elements: the identity, rotations of $90^\circ$, $180^\circ$ and $270^\circ$, two reflections over corner vertices, and 2 reflections over the edge $P_n$. Consequently, $$8\times\Phi_k (P_n^2)=  N_k (P_n^2),$$ where $N_k (P_n^2)$ is the total number of (not necessarily non-equivalent) distinguishing colorings of $P_n^2$.
		
		It is quite time-taking to present and check an inclusion-exclusion argument for $\Phi_k(P_n^2)$. To make it easier, let $N_k (P_n^2)=k^{n^2}-R_k(P_n^2)-S_k(P_n^2)$ where $R_k(P_n^2)$ is the number of colorings of $G$ that are preserved by a rotation and $S_k(P_n^2)$ is the number of colorings of $G$ that are preserved by a reflection but not by a rotation. Therefore, by inclusion-exclusion we have $$R_k(P_n^2)=k^{\lceil \frac{n^2}{2}\rceil} +2k^{\lceil \frac{n^2}{4}\rceil}-3k^{\lceil \frac{n^2}{4}\rceil}+k^{\lceil \frac{n^2}{4}\rceil}=k^{\lceil \frac{n^2}{2}\rceil}.$$
		
		To calculate $S_k(P_n^2)$, first note that there are $2k^{n\lceil\frac{n}{2}\rceil}$ colorings that are preserved by vertical and horizontal reflections over $P_n$, and there are $2k^{n\left(\frac{n+1}{2}\right)}$ that are preserved by reflections over corner vertices. We must also note that a coloring that is preserved by rotations of $90^\circ$ or $270^\circ$ is also preserved by the rotation of $180^\circ$. Therefore, we need only subtract those colorings that are preserved by a reflection and the rotation of $180^\circ$. There are $2k^{\lceil\frac{n}{2}\rceil^2}$ colorings that are preserved by reflections over $P_n$ and the rotation of $180^\circ$, and there are $2k^{\lceil \frac{n}{2}\rceil \lceil \frac{n+1}{2}\rceil}$ colorings that are preserved by reflections over corner vertices and the rotation of $180^\circ$.
		
		We also observe that a coloring that is preserved by two different reflections is also preserved by a rotation, see Figures \ref{reflection-edge} that illustrates this fact for $n=4$. That is why we do not need to add them to our inclusion-exclusion procedure because they will be eliminated during the next rounds. Now the proof is completed.
	\end{proof}
	
	\begin{figure}
		\[ \xygraph{
			!{<0cm,0cm>; <1cm,0cm>:<0cm,1cm>::}
			!{(1,1)}*{\bullet}="v1" !{(2,1)}*{\color[rgb]{0,0,1}\bullet}="v2" !{(3,1)}*{\color[rgb]{0,0,1}\bullet}="v3" !{(4,1)}*{\bullet}="v4"
			!{(1,2)}*{\color[rgb]{0,1,0}\bullet}="v6" !{(2,2)}*{\color[rgb]{1,0,0}\bullet}="v7" !{(3,2)}*{\color[rgb]{1,0,0}\bullet}="v8" !{(4,2)}*{\color[rgb]{0,1,0}\bullet}="v9"
			!{(1,3)}*{\color[rgb]{0,1,0}\bullet}="v11" !{(2,3)}*{\color[rgb]{1,0,0}\bullet}="v12" !{(3,3)}*{\color[rgb]{1,0,0}\bullet}="v13" !{(4,3)}*{\color[rgb]{0,1,0}\bullet}="v14"
			!{(1,4)}*{\bullet}="v16" !{(2,4)}*{\color[rgb]{0,0,1}\bullet}="v17" !{(3,4)}*{\color[rgb]{0,0,1}\bullet}="v18" !{(4,4)}*{\bullet}="v19"
			!{(7,1)}*{\color[rgb]{1,0,1}\bullet}="v21" !{(8,1)}*{\color[rgb]{0,1,0}\bullet}="v22" !{(9,1)}*{\color[rgb]{0,0,1}\bullet}="v23" !{(10,1)}*{\bullet}="v24"
			!{(7,2)}*{\color[rgb]{0,1,0}\bullet}="v26" !{(8,2)}*{\color{orange}\bullet}="v27" !{(9,2)}*{\color[rgb]{1,0,0}\bullet}="v28" !{(10,2)}*{\color[rgb]{0,0,1}\bullet}="v29"
			!{(7,3)}*{\color[rgb]{0,0,1}\bullet}="v31" !{(8,3)}*{\color[rgb]{1,0,0}\bullet}="v32" !{(9,3)}*{\color{orange}\bullet}="v33" !{(10,3)}*{\color[rgb]{0,1,0}\bullet}="v34"
			!{(7,4)}*{\bullet}="v36" !{(8,4)}*{\color[rgb]{0,0,1}\bullet}="v37" !{(9,4)}*{\color[rgb]{0,1,0}\bullet}="v38" !{(10,4)}*{\color[rgb]{1,0,1}\bullet}="v39"
			!{(13,1)}*{\bullet}="v41" !{(14,1)}*{\color[rgb]{0,0,1}\bullet}="v42" !{(15,1)}*{\color[rgb]{0,0,1}\bullet}="v43" !{(16,1)}*{\bullet}="v44"
			!{(13,2)}*{\color[rgb]{0,0,1}\bullet}="v46" !{(14,2)}*{\color[rgb]{1,0,0}\bullet}="v47" !{(15,2)}*{\color[rgb]{1,0,0}\bullet}="v48" !{(16,2)}*{\color[rgb]{0,0,1}\bullet}="v49"
			!{(13,3)}*{\color[rgb]{0,0,1}\bullet}="v51" !{(14,3)}*{\color[rgb]{1,0,0}\bullet}="v52" !{(15,3)}*{\color[rgb]{1,0,0}\bullet}="v53" !{(16,3)}*{\color[rgb]{0,0,1}\bullet}="v54"
			!{(13,4)}*{\bullet}="v56" !{(14,4)}*{\color[rgb]{0,0,1}\bullet}="v57" !{(15,4)}*{\color[rgb]{0,0,1}\bullet}="v58" !{(16,4)}*{\bullet}="v59" 
			"v1"-@[black] "v2" "v2"-@[black] "v3" "v3"-@[black] "v4" 
			"v6"-@[black] "v7" "v7"-@[black] "v8" "v8"-@[black] "v9" 
			"v11"-@[black] "v12" "v12"-@[black] "v13" "v13"-@[black] "v14" 
			"v16"-@[black] "v17" "v17"-@[black] "v18" "v18"-@[black] "v19" 
			"v1"-@[black] "v6" "v6"-@[black] "v11" "v11"-@[black] "v16"
			"v2"-@[black] "v7" "v7"-@[black] "v12" "v12"-@[black] "v17"
			"v3"-@[black] "v8" "v8"-@[black] "v13" "v13"-@[black] "v18"
			"v4"-@[black] "v9" "v9"-@[black] "v14" "v14"-@[black] "v19"
			"v21"-@[black] "v22" "v22"-@[black] "v23" "v23"-@[black] "v24" 
			"v26"-@[black] "v27" "v27"-@[black] "v28" "v28"-@[black] "v29" 
			"v31"-@[black] "v32" "v32"-@[black] "v33" "v33"-@[black] "v34" 
			"v36"-@[black] "v37" "v37"-@[black] "v38" "v38"-@[black] "v39" 
			"v21"-@[black] "v26" "v26"-@[black] "v31" "v31"-@[black] "v36"
			"v22"-@[black] "v27" "v27"-@[black] "v32" "v32"-@[black] "v37"
			"v23"-@[black] "v28" "v28"-@[black] "v33" "v33"-@[black] "v38"
			"v24"-@[black] "v29" "v29"-@[black] "v34" "v34"-@[black] "v39"
			"v41"-@[black] "v42" "v42"-@[black] "v43" "v43"-@[black] "v44" 
			"v46"-@[black] "v47" "v47"-@[black] "v48" "v48"-@[black] "v49" 
			"v51"-@[black] "v52" "v52"-@[black] "v53" "v53"-@[black] "v54" 
			"v56"-@[black] "v57" "v57"-@[black] "v58" "v58"-@[black] "v59" 
			"v41"-@[black] "v46" "v46"-@[black] "v51" "v51"-@[black] "v56"
			"v42"-@[black] "v47" "v47"-@[black] "v52" "v52"-@[black] "v57"
			"v43"-@[black] "v48" "v48"-@[black] "v53" "v53"-@[black] "v58"
			"v44"-@[black] "v49" "v49"-@[black] "v54" "v54"-@[black] "v59"
		} \]
		\caption{(Left) A coloring of $P_4 \Box P_4$ that is preserved by the two reflections over $P_4$. Any such coloring is also preserved by the rotation of $180^\circ$. (Middle) A coloring of $P_4 \Box P_4$ that is preserved by the two reflections over corner vertices. Any such coloring is also preserved by the rotation of $180^\circ$. (Right) A coloring of $P_4 \Box P_4$ that is preserved by the reflection over vertical $P_4$ and the reflection over down-left  and up-right corner vertices. Any such coloring is also preserved by the rotation of $90^\circ$.}\label{reflection-edge}
	\end{figure}
	
	When the index $\Phi_k(G)$ is known for $k=D(G),D(G)+1,\ldots$, we can calculate $\varphi_k(G)$ using Equation \ref{Phi-phi}. It gives us the following recursive formula:
	\begin{equation}
		\varphi_k(G)=\Phi_k(G)-\sum_{i=D(G)}^{k-1} {k\choose i}\varphi_{i}(G).
	\end{equation} 
	
	Therefore, at least for paths, the arisen questions in the introduction are fully answered.
	
	
	\section*{Statements and Declarations}
	The authors declare no conflict of interest and no data are associated with this article.
	
	\bibliographystyle{plain}
	\bibliography{bibliography}

\end{document}